\documentclass[a4paper,10pt]{amsart}

\usepackage{amsmath}
\usepackage{amssymb}
\usepackage{graphicx}
\usepackage{xcolor}

\usepackage{amsthm}
\theoremstyle{definition}
\newtheorem{proposition}{Proposition}[section]
\newtheorem{theorem}{Theorem}[section]
\newtheorem{corollary}{Corollary}[section]
\newtheorem{lemma}{Lemma}[section]

\newtheorem{defn}{Definition}[section]
\newtheorem{remark}{Remark}[section]

\theoremstyle{remark}
\newtheorem*{question}{Question}

\begin{document}

\title[Zero-divisor graphs and zero-divisor real rank]{Zero-divisor graphs and zero-divisor real rank of $C^*$-algebras}

\author{Ali Taghavi}
\address{Department of  Mathematics, Qom University of Technology, Qom, Iran.}
\email{taghavi@qut.ac.ir}

\date{\today}



\begin{abstract}
We study two invariants of $C^*$-algebras built from positive zero
divisors. First, we associate with a $C^*$-algebra $A$ the graph
$\Gamma_+(A)$ whose vertices are the nonzero positive zero divisors of
$A$, two vertices being adjacent when they are orthogonal. We show that
$\Gamma_+$ of the Calkin algebra is connected of diameter $3$, that
$\Gamma_+(B(H))$ is connected of diameter $4$ (for $\dim H\ge3$), that
$\Gamma_+(A)$ is connected of diameter at most $4$ for every von Neumann
factor $A\neq\mathbb C,M_2(\mathbb C)$, and of diameter at most $3$ if
$A$ is not of type~I, whereas $\Gamma_+(M_2(\mathbb C))$ is
disconnected. Second, we study the zero-divisor real rank $ZDRR(A)$,
obtained by replacing invertibility in the definition of real rank with
the weaker condition of being a non-zero-divisor, so that
$ZDRR(A)\le RR(A)$. For commutative algebras $C(X)$ we prove
$ZDRR(C(X))=RR(C(X))=\dim X$ when $X$ is a compact almost-$P$-space,
while $ZDRR(C(X))=0$ for separable $X$. Using Levy's connected, totally
ordered almost-$P$-space we construct unital commutative
$C^*$-algebras $A_n$ with $ZDRR(A_n)=n$ for every $n\in\mathbb N_0$.
On the other hand $ZDRR(C(X\times[0,1]))=0$ for every compact Hausdorff
space $X$, so $ZDRR$ and $RR$ can differ arbitrarily.
\end{abstract}

\maketitle

\section{Introduction}

For a commutative ring $R$, the zero-divisor graph $\Gamma(R)$ of
Anderson and Livingston \cite{AndersonLivingston1999} has the nonzero
zero divisors as vertices, two of them being adjacent when their product
is zero; it is connected and has diameter at most $3$. In a
$C^*$-algebra the natural analogue is obtained by restricting to
\emph{positive} elements, for which $ab=0$ is equivalent to orthogonality
and is automatically symmetric. For a $C^*$-algebra $A$ let $Z_+(A)$ be
the set of nonzero positive zero divisors, and let $\Gamma_+(A)$ be the
graph with vertex set $Z_+(A)$ in which two distinct vertices $a,b$ are
joined when $ab=0$.

In Section~2 we determine the connectivity and diameter of
$\Gamma_+(A)$ for several classes of $C^*$-algebras
(Theorem~\ref{thm:diam}, Corollary~\ref{cor:factors}): the graph of the
Calkin algebra $B(H)/K(H)$ is connected of diameter $3$, the graph of
$B(H)$ is connected of diameter exactly $4$ when $\dim H\ge3$, and the
graph of a von Neumann factor other than $\mathbb C$ and $M_2(\mathbb C)$
is connected of diameter at most $4$, and of diameter at most $3$ if the
factor is not of type~I. The graph of $M_2(\mathbb C)$ is disconnected.
In particular the bound $3$ valid for commutative rings does not hold
for positive zero-divisor graphs of $C^*$-algebras.

In Section~3 we turn to the zero-divisor real rank $ZDRR(A)$. For a
unital $C^*$-algebra $A$ and $n\ge0$, $ZDRR(A)\le n$ means that every
$(n+1)$-tuple of self-adjoint elements can be approximated by a tuple
$(b_1,\dots,b_{n+1})$ for which $b_1^2+\cdots+b_{n+1}^2$ is a
non-zero-divisor. This is the definition of real rank with invertibility
replaced by the weaker condition of being a non-zero-divisor, so
$ZDRR(A)\le RR(A)$. For $A=C(X)$ a function is a non-zero-divisor
precisely when its zero set has empty interior, and one is led to ask for
which compact Hausdorff spaces $X$ the equality
$ZDRR(C(X))=RR(C(X))=\dim X$ holds. It fails for separable $X$ (where
$ZDRR(C(X))=0$, Proposition~\ref{prop:sep}), and we prove that it holds
for every compact almost-$P$-space (Theorem~\ref{thm:almostP}). Combined
with Levy's example of a connected, totally ordered almost-$P$-space
\cite{Levy1977}, the stability of the almost-$P$ property under finite
products (Lemma~\ref{lem:prod}) and Fedorchuk's dimension
formula for products with a linearly ordered continuum
\cite{Fedorchuk1998}, this yields commutative unital $C^*$-algebras
$A_n$ with $ZDRR(A_n)=n$ for every $n$.
In contrast, tensoring with $C[0,1]$ destroys the rank: we show that
$ZDRR(C(X\times[0,1]))=0$ for every compact Hausdorff space $X$
(Theorem~\ref{thm:interval}), so that $ZDRR$ and $RR$ can differ
arbitrarily (Corollary~\ref{cor:gap}).

\medskip
\noindent\textbf{Relation to earlier work.} The graph $\Gamma_+(A)$, the
invariant $ZDRR$ and the set $\mathcal I(A)$ of orthogonal positive sums
were introduced in the author's earlier note \cite{Taghavi2013}. That
note contains the following results relevant here: the Calkin algebra
graph has diameter $3$; the graph of $B(H)$ was also asserted to have
diameter $3$; for a factor of dimension at least $3$ the graph is
connected, with a diameter bound of $6$ proved in the body; the graph of
$M_2(\mathbb C)$ is disconnected; $ZDRR(C(X))=0$ for separable $X$ and for
the lexicographic square; and it was asked (Question~5 there) whether
every value of $ZDRR$ occurs. The present paper differs from
\cite{Taghavi2013} in the following ways.
\begin{itemize}
\item The statement of \cite{Taghavi2013} that $\Gamma_+(B(H))$ has
diameter $3$ is incorrect: we show that the diameter is $4$ whenever
$\dim H\ge3$ (Theorem~\ref{thm:diam}).
\item The diameter bound for factors is improved from $6$ to $4$, and to
$3$ for factors that are not of type~I (Corollary~\ref{cor:factors}).
\item We prove that $ZDRR(C(X))=RR(C(X))=\dim X$ for every compact
almost-$P$-space (Theorem~\ref{thm:almostP}), and we answer the
realization question of \cite{Taghavi2013} for every finite value:
$ZDRR(C(X^n))=n$ (Section~3.1).
\item We show that $ZDRR(C(X\times[0,1]))=0$ for every compact Hausdorff
space $X$ (Theorem~\ref{thm:interval}), which has no counterpart in
\cite{Taghavi2013}.
\end{itemize}
Full proofs are given here for the Calkin algebra and the other graph
results, and for the separable case of $ZDRR$. The set $\mathcal I(A)$
is studied in a separate paper.

\section{Zero-divisor graph of positive elements}

Let $A$ be a $C^*$-algebra. An element $x\in A$ is called a \emph{zero divisor} if $x\neq0$ and $xy=0$ or $yx=0$ for some nonzero $y\in A$. For $x\geq0$ this is equivalent to the existence of a nonzero \emph{positive} $z\in A$ with $xz=0$: if $xy=0$ take $z=yy^*$; if $yx=0$ then $xy^*=0$ and we take $z=y^*y$. We denote by $Z_{+}(A)$ the set of all
positive zero divisors of $A$. Clearly the graph defined below is invariant under $*$-isomorphisms. We associate to $A$ the
undirected graph $\Gamma_{+}(A)$ whose vertices are the elements of
$Z_{+}(A)$, with two distinct vertices $a$ and $b$ joined by an edge
if and only if
\[
ab=0.
\]
Since positive elements are self-adjoint, $ab=0$ implies $ba=0$.

A graph is connected if every two vertices can be joined by a path.
For two vertices $a$ and $b$, we denote by $d(a,b)$ the length of a
shortest path joining them. The diameter of a connected graph
$\Gamma$, denoted by $\operatorname{diam}\Gamma$, is
\[
\operatorname{diam}\Gamma
=
\sup\{d(a,b):a,b\in V(\Gamma)\}.
\]

We begin with two elementary Hilbert-space lemmas.

\begin{lemma}
Let $H_{1}$ and $H_{2}$ be infinite-dimensional closed subspaces of
a Hilbert space $H$. Then there exist infinite-dimensional closed
subspaces
\[
L\subseteq H_{1},
\qquad
F\subseteq H_{2},
\]
such that
\[
L\perp F.
\]
\end{lemma}

\begin{proof}
We construct orthonormal sequences $(e_n)$ in $H_{1}$ and $(f_n)$
in $H_{2}$ inductively.

Choose a unit vector $e_{1}\in H_{1}$. Since $H_{2}$ is
infinite-dimensional, we can choose a unit vector
\[
f_{1}\in H_{2}\cap e_{1}^{\perp}.
\]

Suppose that orthonormal vectors
\[
e_{1},\ldots,e_{n}
\quad\text{and}\quad
f_{1},\ldots,f_{n}
\]
have been chosen so that
\[
\langle e_i,f_j\rangle=0
\qquad (1\leq i,j\leq n).
\]
Since the span of $e_{1},\ldots,e_{n},f_{1},\ldots,f_{n}$ is finite-dimensional and
$H_{1}$ is infinite-dimensional, we may choose a unit vector
\[
e_{n+1}\in H_{1}\cap
\operatorname{span}\{e_{1},\ldots,e_{n},f_{1},\ldots,f_{n}\}^{\perp}.
\]
Similarly, since the span of $e_{1},\ldots,e_{n+1},f_{1},\ldots,f_{n}$ is
finite-dimensional and $H_{2}$ is infinite-dimensional, we may
choose a unit vector
\[
f_{n+1}\in H_{2}\cap
\operatorname{span}\{e_{1},\ldots,e_{n+1},f_{1},\ldots,f_{n}\}^{\perp}.
\]
Thus the two sequences remain orthonormal and every $e_i$ is
orthogonal to every $f_j$.

Let
\[
L=\overline{\operatorname{span}}\{e_n:n\geq1\},
\qquad
F=\overline{\operatorname{span}}\{f_n:n\geq1\}.
\]
Then $L$ and $F$ are infinite-dimensional closed subspaces of
$H_{1}$ and $H_{2}$, respectively, and $L\perp F$.
\end{proof}

\begin{lemma}
Let $H_{0}$ be an infinite-dimensional closed subspace of a Hilbert
space $H$ such that $H_{0}^{\perp}$ is also infinite-dimensional.
Let $p_{H_{0}}$ denote the orthogonal projection onto $H_{0}$. Then
the image of $p_{H_{0}}$ in the Calkin algebra
\[
B(H)/K(H)
\]
is a nonzero positive zero divisor.
\end{lemma}

\begin{proof}
Since $H_{0}$ is infinite-dimensional, the orthogonal projection
$p_{H_{0}}$ is not compact. Indeed, if it were compact, its
restriction to the unit ball of $H_{0}$ would have relatively
compact image, whereas that restriction is the identity operator on
the infinite-dimensional space $H_{0}$. Hence the image
\[
\pi(p_{H_{0}})
\]
of $p_{H_{0}}$ in $B(H)/K(H)$ is nonzero and positive.

Since $H_{0}^{\perp}$ is infinite-dimensional, the orthogonal
projection
\[
1-p_{H_{0}}=p_{H_{0}^{\perp}}
\]
is also noncompact. Consequently,
\[
\pi(1-p_{H_{0}})
\]
is a nonzero positive element of the Calkin algebra. Moreover,
\[
\pi(p_{H_{0}})\pi(1-p_{H_{0}})
=
\pi\bigl(p_{H_{0}}(1-p_{H_{0}})\bigr)
=
0.
\]
Thus $\pi(p_{H_{0}})$ is a nonzero positive zero divisor.
\end{proof}

We now apply these lemmas to the Calkin algebra and to $B(H)$.

We shall also need the following lifting lemma, which allows orthogonality
relations in a quotient $A/J$ to be lifted to $A$.

\begin{lemma}\label{lem:lift}
Let $A$ be a $C^*$-algebra, let $J$ be a closed two-sided ideal of $A$,
let $\pi:A\to A/J$ be the quotient map, and let $e,f\in(A/J)_+$ satisfy
$ef=0$. Let $x,y\in A_+$ with $\pi(x)=e$ and $\pi(y)=f$, and put
\[
z=(x-y)_+,\qquad w=(x-y)_-.
\]
Then $z,w\in A_+$, $zw=0$, $\pi(z)=e$, $\pi(w)=f$, $\|z\|\leq\|x\|$ and
$\|w\|\leq\|y\|$.
\end{lemma}

\begin{proof}
That $z,w\geq0$ and $zw=0$ is standard. By the compatibility of the
continuous functional calculus with $*$-homomorphisms,
$\pi(z)=(\pi(x-y))_+=(e-f)_+$ and $\pi(w)=(e-f)_-$. Since $e,f\geq0$
commute and $ef=0$, every character $\chi$ of the commutative
$C^*$-algebra $C^*(e,f)$ satisfies $\chi(e)\chi(f)=0$, so
$(\chi(e)-\chi(f))_+=\chi(e)$ and $(\chi(e)-\chi(f))_-=\chi(f)$. Hence
$(e-f)_+=e$ and $(e-f)_-=f$. Finally, $x-y\leq x$ gives
$\sigma(x-y)\subseteq(-\infty,\|x\|]$, hence $\|z\|\leq\|x\|$, and
$y-x\leq y$ gives $\|w\|=\|(y-x)_+\|\leq\|y\|$.
\end{proof}

\begin{theorem}\label{thm:diam}
Let $H$ be an infinite-dimensional Hilbert space. Then
\[
\Gamma_{+}(B(H))
\quad\text{and}\quad
\Gamma_{+}\bigl(B(H)/K(H)\bigr)
\]
are connected graphs and
\[
\operatorname{diam}\Gamma_{+}(B(H))=4,
\qquad
\operatorname{diam}\Gamma_{+}\bigl(B(H)/K(H)\bigr)=3.
\]
\end{theorem}

\begin{proof}
Put
\[
A=B(H)/K(H)
\]
and let
\[
\pi:B(H)\longrightarrow A
\]
be the quotient map.

Let $S,T\in Z_{+}(A)$ be arbitrary. Since $S$ and $T$ are positive
zero divisors, there exist nonzero positive elements
\[
S_{1},T_{1}\in A
\]
such that
\[
SS_{1}=0,
\qquad
TT_{1}=0.
\]
By Lemma~\ref{lem:lift} (applied to positive lifts of $S,S_1$ and of $T,T_1$), there exist nonzero positive operators
\[
\widetilde S,\widetilde S_{1},
\widetilde T,\widetilde T_{1}\in B(H)
\]
lifting respectively $S,S_{1},T,T_{1}$ and satisfying
\[
\widetilde S\widetilde S_{1}=0,
\qquad
\widetilde T\widetilde T_{1}=0.
\]
None of these four operators is compact, since their images in the
Calkin algebra are nonzero.

It follows that
\[
\operatorname{ran}\widetilde S_{1}
\subseteq
\ker\widetilde S,
\qquad
\operatorname{ran}\widetilde T_{1}
\subseteq
\ker\widetilde T.
\]
A noncompact operator cannot have finite-dimensional range. Hence
\[
\operatorname{ran}\widetilde S_{1}
\quad\text{and}\quad
\operatorname{ran}\widetilde T_{1}
\]
are infinite-dimensional, and therefore
\[
\ker\widetilde S
\quad\text{and}\quad
\ker\widetilde T
\]
are infinite-dimensional.

By the first lemma, there exist infinite-dimensional closed subspaces
\[
L\subseteq\ker\widetilde S,
\qquad
F\subseteq\ker\widetilde T
\]
such that
\[
L\perp F.
\]
Let $p_L$ and $p_F$ denote the orthogonal projections onto $L$ and
$F$, respectively. Since $L$ and $F$ are infinite-dimensional, both
projections are noncompact. Moreover,
\[
\widetilde S p_L=0,
\qquad
p_Lp_F=0,
\qquad
p_F\widetilde T=0.
\]

Let
\[
\pi_L=\pi(p_L),
\qquad
\pi_F=\pi(p_F)
\]
denote their images in the Calkin algebra.

Because $F\subseteq L^\perp$ and $L\subseteq F^\perp$, both
$L^\perp$ and $F^\perp$ are infinite-dimensional. Hence, by the
second lemma, $\pi_L$ and $\pi_F$ are nonzero positive zero divisors
in $A$. Furthermore,
\[
S\pi_L=0,
\qquad
\pi_L\pi_F=0,
\qquad
\pi_FT=0.
\]
Thus
\[
S,\ \pi_L,\ \pi_F,\ T
\]
gives a walk of length $3$ from $S$ to $T$ (consecutive entries are distinct, because a positive element whose square vanishes is zero). Hence
\[
\operatorname{diam}\Gamma_{+}(A)\leq3.
\]

We now show that the diameter is strictly greater than $2$.

Choose an orthogonal decomposition
\[
H=E\oplus F\oplus G,
\]
where $E,F,G$ are infinite-dimensional closed subspaces. Define
\[
H_{0}=E\oplus F,
\qquad
H_{1}=E\oplus G,
\qquad
H_{2}=G.
\]
Each of these subspaces is infinite-dimensional and has
infinite-dimensional orthogonal complement.

Let $p_0,p_1,p_2$ denote the corresponding orthogonal projections.
By the second lemma, the images
\[
S=\pi(p_0),
\qquad
T=\pi(p_1)
\]
are nonzero positive zero divisors in the Calkin algebra.

We have
\[
H_{0}\cap H_{1}=E,
\qquad
H_{0}+H_{1}=H,
\qquad
H_{2}\subseteq H_{1}.
\]
Consequently,
\[
1=p_0+p_2
\]
and
\[
p_1p_2=p_2p_1=p_2.
\]

Suppose that $G_0\in B(H)$ satisfies
\[
G_0p_0\in K(H)
\qquad\text{and}\qquad
G_0p_1\in K(H).
\]
Then
\[
G_0p_2
=
G_0p_1p_2
\]
is compact, since $K(H)$ is a two-sided ideal in $B(H)$. Hence
\[
G_0
=
G_0(p_0+p_2)
=
G_0p_0+G_0p_2
\]
is compact.

On the other hand,
\[
p_0p_1
\]
is the orthogonal projection onto the infinite-dimensional space
\[
H_{0}\cap H_{1}=E,
\]
and is therefore not compact. Hence
\[
ST=\pi(p_0p_1)\neq0.
\]

Now suppose that there were a nonzero positive element $G_0\in A$
satisfying
\[
G_0S=G_0T=0.
\]
Choose a positive lift $\widetilde G_0\in B(H)$ of $G_0$. Then
\[
\widetilde G_0p_0
\quad\text{and}\quad
\widetilde G_0p_1
\]
are compact, since their images in the Calkin algebra are
respectively $G_0S=0$ and $G_0T=0$. By the preceding argument,
$\widetilde G_0$ is compact. Thus $G_0=0$, a contradiction.

Therefore $S$ and $T$ have no common nonzero positive neighbor.
Since $ST\neq0$, they are not adjacent, and since they have no
common neighbor, they cannot be joined by a path of length $2$.
Hence
\[
d(S,T)>2.
\]
Together with the upper bound, this gives
\[
\operatorname{diam}\Gamma_{+}(A)=3.
\]

It remains to prove the assertion for $B(H)$. In this part $H$ may be
any Hilbert space with $\dim H\geq3$, finite or infinite; we shall use
this in the proof of Corollary~\ref{cor:factors}. Note that a positive
operator $S\in B(H)$ is a zero divisor if and only if
$\ker S\neq\{0\}$. Indeed, if $0\neq\xi\in\ker S$, the rank-one
projection $p_\xi$ onto $\mathbb C\xi$ satisfies $Sp_\xi=0$;
conversely, if $0\neq R\in B(H)$ satisfies $SR=0$, then
$\operatorname{ran}R\subseteq\ker S$. Moreover, $Sp=0$ for every
projection $p$ onto a subspace of $\ker S$.

Let $S,T\in Z_{+}(B(H))$ and put
\[
M=\ker S,\qquad N=\ker T,
\]
both nonzero closed subspaces.

\emph{Step 1.} If $M\cap N\neq\{0\}$, let $p$ be the projection onto
$M\cap N$. Then $Sp=pT=0$, so $S,p,T$ is a path of length at most $2$.

\emph{Step 2.} Suppose there are unit vectors $e\in M$, $f\in N$ with
$e\perp f$, and let $p_e$, $p_f$ be the rank-one projections onto
$\mathbb Ce$, $\mathbb Cf$. Then
\[
Sp_e=0,\qquad p_ep_f=0,\qquad p_fT=0,
\]
so $S,p_e,p_f,T$ is a path of length at most $3$.

\emph{Step 3.} Suppose that neither Step~1 nor Step~2 applies. If
$\dim M\geq2$, pick a unit vector $f\in N$; then $M\cap f^\perp$ has
codimension at most $1$ in $M$, hence contains a unit vector $e$, and
Step~2 applies. The same holds if $\dim N\geq2$. Hence
$M=\mathbb Ce$ and $N=\mathbb Cf$ are one-dimensional, with $e,f$ unit
vectors which are neither orthogonal nor proportional. Let $q$ be the
projection onto $\{e,f\}^\perp$, which is nonzero because
$\dim H\geq3$. Then
\[
Sp_e=0,\qquad p_eq=0,\qquad qp_f=0,\qquad p_fT=0,
\]
so $S,p_e,q,p_f,T$ is a path of length $4$.

Thus $\Gamma_{+}(B(H))$ is connected and
$\operatorname{diam}\Gamma_{+}(B(H))\leq4$.

To see that the diameter is at least $4$, choose two linearly
independent, non-orthogonal unit vectors $e,f\in H$, let $p_e,p_f$ be the
rank-one projections onto $\mathbb Ce,\mathbb Cf$, and put
\[
S=1-p_e,\qquad T=1-p_f.
\]
Then $S,T\in Z_{+}(B(H))$, since $\ker S=\mathbb Ce$ and
$\ker T=\mathbb Cf$. If $G\geq0$ and $GS=0$, then $G$ vanishes on
$\operatorname{ran}S=e^\perp$, so $G=Gp_e=p_eG$ and hence $G=\lambda p_e$
with $\lambda>0$ (for $G\neq0$). Thus the neighbours of $S$ are exactly the operators
$\lambda p_e$ ($\lambda>0$), and likewise the neighbours of $T$ are exactly
the operators $\mu p_f$ ($\mu>0$). Now:
\begin{itemize}
\item $S$ and $T$ are not adjacent: $ST=0$ would give
$f^\perp=\operatorname{ran}T\subseteq\ker S=\mathbb Ce$, which is
impossible because $\dim f^\perp\geq2$;
\item they have no common neighbour, since $\lambda p_e\neq\mu p_f$;
\item there is no path $S,x,y,T$ of length $3$: it would need
$x=\lambda p_e$, $y=\mu p_f$ and $xy=0$, i.e.\ $p_ep_f=0$, i.e.\
$e\perp f$, which was excluded.
\end{itemize}
Hence $d(S,T)\geq4$. Together with the upper bound, this gives
$\operatorname{diam}\Gamma_{+}(B(H))=4$.
\end{proof}

\begin{corollary}\label{cor:factors}
Let $A$ be a von Neumann factor. If $A\not\simeq M_{2}(\mathbb C)$
and $A\neq\mathbb C$, then $\Gamma_{+}(A)$ is connected and
\[
\operatorname{diam}\Gamma_{+}(A)\leq4.
\]
If, in addition, $A$ is not of type~I (equivalently, $A$ has no minimal
projection), then $\operatorname{diam}\Gamma_{+}(A)\leq3$.
\end{corollary}

\begin{proof}
A type~I factor is $*$-isomorphic to $B(K)$ for some Hilbert space $K$,
and $\dim K\geq3$ because $A\neq\mathbb C$ and $A\not\simeq M_2(\mathbb C)$.
Since $\Gamma_{+}$ is invariant under $*$-isomorphisms, the part of the
proof of Theorem~\ref{thm:diam} concerning $B(H)$, which is valid for every
Hilbert space of dimension at least $3$, shows that
$\Gamma_{+}(A)$ is connected with
$\operatorname{diam}\Gamma_{+}(A)=4$.

It remains to consider a factor $A$ which is not of type~I. Then $A$ has no
minimal projections, and we regard $A$ as a von Neumann algebra acting on a
Hilbert space $H$. We use the standard fact that, for every $R\in A$, the
orthogonal projections onto $\ker R$ and onto
$\overline{\operatorname{ran}R}$ belong to $A$; see \cite[p.~19]{OLSEN}.
Recall also that the projections of $A$ form a complete lattice, so that
meets and joins of projections of $A$ are again in $A$.

Let $S,T\in Z_{+}(A)$ and let $p,q\in A$ be the projections onto $\ker S$
and $\ker T$. Then $p,q\neq0$ (since $S,T$ are zero divisors) and
$p,q\neq1$ (since $S,T\neq0$). Every projection $e\leq p$ satisfies $Se=0$,
and every projection $e\leq q$ satisfies $Te=0$. We show that
$d(S,T)\leq3$.

\emph{Case A: $p\wedge(1-q)\neq0$.} Put $e=p\wedge(1-q)$. Then
$Se=0$, $eq=0$ and $qT=0$, so $S,e,q,T$ is a path of length at most $3$.

\emph{Case B: $p\wedge(1-q)=0$.} Put $a=pqp$, a positive element of the von
Neumann algebra $pAp$ (with unit $p$). We claim that there are nonzero
orthogonal projections $e,e'\in pAp$ which commute with $a$. If
$\sigma(a)$, computed in $pAp$, contains at least two points, choose
disjoint open sets $U_1,U_2\subset\mathbb R$ each meeting $\sigma(a)$ and
let $e=\chi_{U_1}(a)$, $e'=\chi_{U_2}(a)$; these are nonzero, orthogonal, and
commute with $a$. If $\sigma(a)=\{\lambda\}$, then $a=\lambda p$ commutes
with everything in $pAp$, and, since $A$ has no minimal projection, we can
write $p=e+e'$ with $e,e'$ nonzero orthogonal projections.

Now $e,e'\leq p$, so $e=ep=pe$ and $e'=pe'=e'p$, whence
\[
eqe'=e\,pqp\,e'=eae'=a\,ee'=0.
\]
Let $f\in A$ be the projection onto $\overline{\operatorname{ran}(qe')}$.
Then $f\leq q$, so $Tf=0$, and $eqe'=0$ means that
$\operatorname{ran}(qe')\subseteq\ker e$, i.e.\ $fe=0$. Moreover
$f\neq0$: if $qe'=0$, then $\operatorname{ran}e'\subseteq\ker q$, so
$e'\leq p\wedge(1-q)=0$, a contradiction. Since $e\leq p$ gives $Se=0$,
we obtain the path $S,e,f,T$ of length at most $3$.

Thus every two vertices of $\Gamma_{+}(A)$ can be joined by a path of
length at most $3$, so $\Gamma_{+}(A)$ is connected and
$\operatorname{diam}\Gamma_{+}(A)\leq3$.
\end{proof}

\begin{proposition}
The graph $\Gamma_{+}(M_{2}(\mathbb C))$ is disconnected.
\end{proposition}

\begin{proof}
Every positive matrix is unitarily equivalent to a diagonal matrix
with nonnegative diagonal entries. A positive zero divisor in
$M_{2}(\mathbb C)$ is therefore of the form
\[
a=\lambda p,
\]
where $\lambda>0$ and $p$ is a rank-one projection.

It is consequently enough to consider the graph of rank-one
projections. A rank-one projection $p$ is the orthogonal projection
onto a one-dimensional subspace $L_p$ of $\mathbb C^2$. For rank-one
projections $p$ and $q$,
\[
pq=0
\]
if and only if
\[
L_p\perp L_q.
\]

In $\mathbb C^2$, if $L_p$ and $L_q$ are distinct and orthogonal,
then $L_q=L_p^\perp$. Hence the connected component containing
$\lambda p$, for $\lambda>0$, consists precisely of the vertices
\[
\{\lambda p:\lambda>0\}
\cup
\{\mu(1-p):\mu>0\}.
\]
Since there are infinitely many one-dimensional subspaces of
$\mathbb C^2$, there are infinitely many such components. Thus
\[
\Gamma_{+}(M_{2}(\mathbb C))
\]
is disconnected.
\end{proof}

\begin{remark}
The preceding results illustrate a distinction between the positive
zero-divisor graph of a $C^*$-algebra and the usual zero-divisor graph
of a commutative ring. Recall that, for a commutative ring $R$, the
zero-divisor graph $\Gamma(R)$ is connected whenever it is nonempty,
and

$$
\operatorname{diam}\Gamma(R)\leq 3;
$$

moreover, this bound is sharp \cite{AndersonLivingston1999}. In contrast, $\Gamma_+(B(H))$ has diameter $4$ (Theorem~\ref{thm:diam}), so the bound $3$ of the commutative ring case fails for the positive zero-divisor graph.

In the present setting, the vertices are restricted to positive
elements. Thus, although the adjacency relation is again given by
orthogonality,

$$
a\sim b\quad\Longleftrightarrow\quad ab=0,
$$

the positivity requirement may exclude intermediate vertices that are
available in the classical zero-divisor graph. On the other hand, in
the Calkin algebra and in $B(H)$, the existence of sufficiently large
orthogonal subspaces allows one to connect arbitrary positive
zero-divisors by short paths. For factors, the presence of projections
in the von Neumann algebra provides a corresponding mechanism.

This comparison naturally raises the following questions: For which
$C^*$-algebras $A$ is $\Gamma_+(A)$ connected? More generally, is
there a universal constant $N$ such that

$$
\operatorname{diam}\Gamma_+(A)\leq N
$$

for every $C^*$-algebra $A$ for which $\Gamma_+(A)$ is connected?
These questions indicate how the order structure of a $C^*$-algebra
can affect the geometry of its positive zero-divisor graph.
\end{remark}

\section{Approximation and rank theory arising from zero-divisors}

The preceding results concern exact zero-divisor configurations. We now
turn to an approximation problem and introduce a rank invariant which
measures how densely non-zero-divisors occur among sums of squares of
self-adjoint elements. Our definition is motivated by the role of real
rank in the theory of $C^*$-algebras. Real rank was introduced by
Brown and Pedersen \cite{BrownPedersen}, while stable rank was developed,
in a different direction, by Rieffel \cite{RIEFFEL}. The invariant
introduced below replaces invertibility in the definition of real rank
by the weaker condition of being a non-zero-divisor.

\begin{defn}
Let $A$ be a unital $C^*$-algebra and let $n\geq0$. We say that
\[
ZDRR(A)\leq n
\]
if, for every $(n+1)$-tuple
\[
(a_1,\ldots,a_{n+1})\in A_{\mathrm{sa}}^{\,n+1}
\]
and every $\varepsilon>0$, there exists a tuple
\[
(b_1,\ldots,b_{n+1})\in A_{\mathrm{sa}}^{\,n+1}
\]
such that
\[
\|a_i-b_i\|<\varepsilon,\qquad i=1,\ldots,n+1,
\]
and
\[
b_1^2+\cdots+b_{n+1}^2
\]
is a non-zero-divisor in $A$, that is, it is not annihilated by any
nonzero element of $A$. Since this element is self-adjoint, the
vanishing of its left and right annihilators are equivalent.
\end{defn}

\begin{defn}
We define
\[
ZDRR(A)=0
\]
if
\[
ZDRR(A)\leq0.
\]
For $n\geq1$, we define
\[
ZDRR(A)=n
\]
if
\[
ZDRR(A)\leq n
\]
and
\[
ZDRR(A)\not\leq n-1.
\]
If there is no finite integer $n$ for which $ZDRR(A)\leq n$, we set
\[
ZDRR(A)=\infty.
\]
\end{defn}

The definition should be compared with the usual definition of real
rank. Recall that $RR(A)\leq n$ if every $(n+1)$-tuple of self-adjoint
elements of $A$ can be approximated arbitrarily closely (equivalently,
within any prescribed coordinatewise norm tolerance) by a tuple
$(b_1,\ldots,b_{n+1})$ such that
\[
b_1^2+\cdots+b_{n+1}^2
\]
is invertible. Thus the distinction between the two invariants is that,
in the definition of $ZDRR$, invertibility is replaced by the weaker
condition of being a non-zero-divisor. Consequently,
\[
ZDRR(A)\leq RR(A).
\]

In the commutative setting this distinction has a direct topological
interpretation. If $X$ is a compact Hausdorff space and
$f\in C(X)$, then $f$ is a non-zero-divisor precisely when its zero set
has empty interior. Hence, for $A=C(X)$, the condition
$ZDRR(C(X))\leq n$ asks whether every $(n+1)$-tuple of real-valued
continuous functions can be approximated by a tuple whose common zero
set has empty interior.

This leads to a natural question connecting the zero-divisor rank with
classical dimension theory.

\begin{question}
For which compact Hausdorff spaces $X$ do we have
\[
ZDRR(C(X))=RR(C(X))?
\]
Equivalently, for which $X$ does
\[
ZDRR(C(X))=\dim X?
\]
\end{question}

The equality in the question does not hold for arbitrary compact
Hausdorff spaces. Indeed, if $X$ is a separable
compact Hausdorff space, then $ZDRR(C(X))=0$ (Proposition~\ref{prop:sep}
below; this fact was first noted in \cite{Taghavi2013}). Thus, for
$X=[0,1]$,

$$
ZDRR(C([0,1]))=0
\qquad\text{while}\qquad
RR(C([0,1]))=\dim[0,1]=1.
$$

Hence the equality $ZDRR(C(X))=RR(C(X))$ is a genuinely restrictive
topological property, and the question above suggests a broader
problem of identifying the compact Hausdorff spaces for which this
equality holds. The theorem below provides a substantial class of such
spaces, namely compact almost-$P$-spaces.

\begin{proposition}\label{prop:sep}
If $X$ is a separable compact Hausdorff space, then $ZDRR(C(X))=0$.
\end{proposition}

\begin{proof}
Let $f\in C(X)$ be real valued and let $\varepsilon>0$. For $y\in\mathbb R$
let $U_y$ be the interior of $f^{-1}(y)$. The sets $U_y$ are pairwise
disjoint and open. If $\{d_n\}$ is a countable dense subset of $X$, every
nonempty $U_y$ contains some $d_n$, so $U_y\neq\varnothing$ for at most
countably many $y$. Choose $t\in(0,\varepsilon)$ with
$U_t=\varnothing$ and put $b=f-t\cdot1$. Then $\|f-b\|=t<\varepsilon$ and the
zero set of $b^2$ is $f^{-1}(t)$, which has empty interior. Hence $b^2$
is a non-zero-divisor, and $ZDRR(C(X))\le0$. (This is the argument of
\cite{Taghavi2013}.)
\end{proof}

Recall that a topological space $X$ is called an almost-$P$-space if
every nonempty $G_\delta$ subset of $X$ has nonempty interior.
The following observation gives the class of spaces needed below.
\begin{theorem}\label{thm:almostP}
Let $X$ be a compact Hausdorff almost-$P$-space. Then
\[
ZDRR(C(X))=RR(C(X))=\dim X.
\]
\end{theorem}

\begin{proof}
For $f\in C(X)_+$, multiplication by $f$ is injective on $C(X)$ if
and only if $Z(f)=f^{-1}(\{0\})$ has empty interior. Indeed, if
$Z(f)$ has nonempty interior, a nonzero continuous function supported
in that interior is annihilated by $f$; conversely, if $fg=0$ for a
nonzero $g$, then the open set $\{x:|g(x)|>0\}$ is contained in
$Z(f)$, so $Z(f)$ has nonempty interior.

Now let $f\in C(X)_+$ be a non-zero-divisor. Its zero set $Z(f)$ is a
closed $G_\delta$ subset of $X$. If $Z(f)$ were nonempty, the
almost-$P$ property would imply that $Z(f)$ has nonempty interior,
contrary to the fact that $f$ is a non-zero-divisor. Hence $Z(f)=\varnothing$,
so $f$ is invertible. The converse is immediate. Thus, on an
almost-$P$-space, a positive element of $C(X)$ is a non-zero-divisor
if and only if it is invertible.

Consequently, in the definition of $ZDRR(C(X))$ the admissible
approximating tuples are exactly the admissible tuples in the
definition of $RR(C(X))$. Therefore
\[
ZDRR(C(X))=RR(C(X)).
\]
Finally, the classical formula for compact Hausdorff spaces gives
$RR(C(X))=\dim X$, and the result follows.
\end{proof}

Here the equivalence follows from the classical formula
\[
RR(C(X))=\dim X
\]
for compact Hausdorff spaces $X$; see, for example,
\cite{BrownPedersen}. Thus the question asks when replacing
invertibility by the weaker non-zero-divisor condition leaves the real
rank unchanged. The results below give a substantial class of positive
answers.

In particular, for the compact almost-$P$-spaces considered below,
\[
ZDRR(C(X))=RR(C(X))=\dim X.
\]

\subsection{Levy's example and arbitrary finite values}

We now obtain examples realizing every positive integer as the
zero-divisor real rank. Let $\mathcal L$ be the Dedekind completion of
the $\eta_1$-set considered in Levy \cite{Levy1977} (this is not the
real line). By \cite[Proposition~2.3]{Levy1977}, $\mathcal L$ is a
connected, totally ordered almost-$P$-space and $\beta\mathcal L$ is its
two-point compactification. Put
\[
X=\beta\mathcal L .
\]
By \cite[Proposition~2.3]{Levy1977}, $X$ is a compact, connected,
totally ordered almost-$P$-space. In particular, $X$ is a nondegenerate
linearly ordered continuum.

We first note that the almost-$P$ property is preserved under finite
products.

\begin{lemma}\label{lem:prod}
If $X_1$ and $X_2$ are almost-$P$-spaces, then so is $X_1\times X_2$.
\end{lemma}

\begin{proof}
Let $G=\bigcap_{k\ge1}U_k$ be a nonempty $G_\delta$ subset of
$X_1\times X_2$, with each $U_k$ open, and let $(x_1,x_2)\in G$. For each
$k$ choose open sets $A_k\ni x_1$ and $B_k\ni x_2$ with
$A_k\times B_k\subseteq U_k$. Then $\bigcap_kA_k$ is a nonempty
$G_\delta$ subset of $X_1$ and $\bigcap_kB_k$ is a nonempty $G_\delta$
subset of $X_2$, so they have nonempty interiors $W_1$ and $W_2$. Then
$W_1\times W_2\subseteq\bigcap_k(A_k\times B_k)\subseteq G$ is a
nonempty open set. \end{proof}

Since $X$ is an almost-$P$-space, it follows inductively that
\[
X^n\text{ is an almost-}P\text{-space}
\qquad (n\geq1).
\]

We next determine the covering dimension of these finite powers.
Fedorchuk proved that if $K$ is a paracompact space and $L$ is a
linearly ordered continuum, then
\[
\dim(K\times L)=\dim K+1;
\]
see \cite[Proposition~3.10]{Fedorchuk1998}. Since every compact
Hausdorff space is paracompact and $X$ is a linearly ordered continuum,
we may apply this result with $K=X^{n-1}$ and $L=X$. Taking $K$ to be a
singleton first gives
\[
\dim X=1.
\]
Consequently,
\[
\dim(X^n)
=
\dim(X^{n-1}\times X)
=
\dim(X^{n-1})+1,
\]
and induction yields
\[
\boxed{\dim(X^n)=n\qquad(n\geq1).}
\]

Since each $X^n$ is a compact almost-$P$-space, Theorem~\ref{thm:almostP}
gives
\[
ZDRR\bigl(C(X^n)\bigr)
=
RR\bigl(C(X^n)\bigr).
\]
Moreover, for every compact Hausdorff space $Y$,
\[
RR(C(Y))=\dim Y.
\]
Therefore,
\[
\begin{aligned}
ZDRR\bigl(C(X^n)\bigr)
&=RR\bigl(C(X^n)\bigr)\\
&=\dim(X^n)\\
&=n,
\end{aligned}
\qquad n\geq1.
\]

Thus we obtain the following.

\begin{theorem}
For every positive integer $n$, there exists a unital commutative
$C^*$-algebra $A_n$ such that
\[
ZDRR(A_n)=n.
\]
More precisely, one may take
\[
A_n=C(X^n),
\]
where $X=\beta\mathcal L$ is Levy's compact connected totally ordered
almost-$P$-space.
\end{theorem}

\begin{remark}
Together with the examples with $ZDRR=0$ discussed above, the theorem
shows that the zero-divisor real rank of unital commutative
$C^*$-algebras assumes every value in $\mathbb{N}_0$. Note that the
equality $ZDRR(C(X))=RR(C(X))=\dim X$ of Theorem~\ref{thm:almostP} already
shows, without any dimension formula, that $ZDRR$ is not identically
zero: the space $X$ above is connected and has more than one point, so
$\dim X\ge1$ and hence $ZDRR(C(X))\ge1$. The dimension formula is used
only to obtain the exact values $\dim X^n=n$.
\end{remark}

\begin{theorem}\label{thm:interval}
Let $X$ be a compact Hausdorff space. Then
\[
ZDRR\bigl(C(X\times[0,1])\bigr)=0 .
\]
\end{theorem}

\begin{proof}
Put $Y=X\times[0,1]$. Let $f\in C(Y)$ be real valued and let
$\varepsilon>0$. For $x\in X$ write $f_x=f(x,\cdot)\in C[0,1]$. Since $Y$
is compact, $x\mapsto f_x$ is continuous from $X$ to $C[0,1]$ with the
supremum norm. For a closed interval $I=[p,q]\subseteq[0,1]$ with
rational endpoints $p<q$, the set
\[
K_I=\{f_x|_I:\ x\in X\}\subseteq C(I)
\]
is therefore compact. Since $C(I)$ is infinite dimensional, $K_I$ has
empty interior. The restriction map $R_I:C[0,1]\to C(I)$ is a bounded
linear surjection, hence open by the open mapping theorem, so the closed
set
\[
B_I=\{v\in C[0,1]:\ R_I v\in -K_I\}
\]
has empty interior (a nonempty open subset of $B_I$ would be mapped by
$R_I$ onto a nonempty open subset of $-K_I$). There are countably many
rational intervals $I$, so by the Baire category theorem there is
$v\in C[0,1]$ with $\|v\|<\varepsilon$ and $v\notin\bigcup_I B_I$. Define
$b(x,t)=f(x,t)+v(t)$. Then $b$ is real valued and
$\|f-b\|=\|v\|<\varepsilon$.

We claim that the zero set of $b$ has empty interior. Otherwise it
contains a set $V\times J$ with $V\subseteq X$ open and nonempty and
$J\subseteq[0,1]$ a nonempty open interval. Choose a rational interval
$I\subseteq J$ and a point $x\in V$. Then $f_x(t)+v(t)=0$ for all
$t\in I$, that is $R_Iv=-f_x|_I\in -K_I$, so $v\in B_I$, a
contradiction. Hence $b^2$ has a zero set with empty interior, and $b^2$
is a non-zero-divisor in $C(Y)$. Therefore $ZDRR(C(Y))\le 0$.
\end{proof}

\begin{corollary}\label{cor:gap}
For every positive integer $n$ there is a compact Hausdorff space $Y_n$
with
\[
ZDRR\bigl(C(Y_n)\bigr)=0
\qquad\text{and}\qquad
RR\bigl(C(Y_n)\bigr)\ge n .
\]
In particular one may take $Y_n=X^n\times[0,1]$, where $X$ is the space
of Section~3.1.
\end{corollary}

\begin{proof}
By Theorem~\ref{thm:interval}, $ZDRR(C(Y_n))=0$. The set
$X^n\times\{0\}$ is closed in $Y_n$ and homeomorphic to $X^n$, and the
covering dimension does not increase on passing to closed subspaces of a
normal space. Hence $RR(C(Y_n))=\dim Y_n\ge\dim X^n=n$.
\end{proof}

\begin{remark}
The unital subalgebra $C(X^n)\subseteq C(X^n\times[0,1])$ has
$ZDRR=n$, while the larger algebra has $ZDRR=0$. Thus $ZDRR$ can drop
when one tensors with $C[0,1]$, whereas $RR(C(Y\times[0,1]))=\dim(Y\times[0,1])$
is never smaller than $RR(C(Y))=\dim Y$.
\end{remark}

\begin{question}
Is $ZDRR(A\otimes C[0,1])=0$ for every unital $C^*$-algebra $A$? Is there
a simple unital noncommutative $C^*$-algebra $A$ with $ZDRR(A)>0$?
\end{question}

\begin{remark}
Since $ZDRR(A)\le RR(A)$, we have $ZDRR(A)=0$ whenever $RR(A)=0$; in
particular $ZDRR$ vanishes for AF algebras, UHF algebras, Cuntz algebras
and von Neumann algebras. A noncommutative example with $ZDRR(A)>0$
therefore requires $RR(A)\ge1$.
\end{remark}

\section*{Statements and Declarations}

\subsection*{Competing Interests}
The author declares that there are no competing interests.

\subsection*{Data Availability}
No datasets were generated or analysed during the current study.

\subsection*{Use of AI Tools}
During the preparation of this manuscript, the author used AI-assisted
tools for language editing, LaTeX formatting, and discussion of
mathematical arguments. The author takes full responsibility for the
content, mathematical results, and final version of the manuscript.

\end{document}